\documentclass[12pt]{amsart}
\usepackage{amsfonts,amssymb,amscd,amstext,mathrsfs}

\input xy
\xyoption{all}

\newtheorem{theorem}{Theorem}
\newtheorem{lemma}{Lemma}

\theoremstyle{definition}
\newtheorem{remark}{Remark}

\newcommand{\C}{\mathbb{C}}
\renewcommand{\P}{\mathbb{P}}

\begin{document}

\title{Approximation and interpolation of regular maps \\ from affine varieties to algebraic manifolds}

\author{Finnur L\'arusson}

\address{Finnur L\'arusson, School of Mathematical Sciences, University of Adelaide, Adelaide SA 5005, Australia}
\email{finnur.larusson@adelaide.edu.au}

\author{Tuyen Trung Truong}
\address{Tuyen Trung Truong, School of Mathematical Sciences, University of Adelaide, Adelaide SA 5005, Australia}
\curraddr{Department of Mathematics, University of Oslo, Blindern, 0316 Oslo, Norway}
\email{tuyentt@math.uio.no}

\thanks{The authors were supported by Australian Research Council grant DP150103442.}

\subjclass[2010]{Primary 14R10.  Secondary 14M20, 14M25, 32E10, 32Q28}

\date{First version 1 June 2017.  This version 4 July 2017.  A few minor edits 1 May 2018}

\keywords{Oka principle, Oka theory, affine variety, algebraic manifold, Stein space, Oka manifold, regular map, holomorphic map, approximation, interpolation}

\begin{abstract} 
We consider the analogue for regular maps from affine varieties to suitable algebraic manifolds of Oka theory for holomorphic maps from Stein spaces to suitable complex manifolds.  The goal is to understand when the obstructions to approximation or interpolation are purely topological.  We propose a definition of an algebraic Oka property, which is stronger than the analytic Oka property.  We review the known examples of algebraic manifolds satisfying the algebraic Oka property and add a new class of examples: smooth nondegenerate toric varieties.  On the other hand, we show that the algebraic analogues of three of the central properties of analytic Oka theory fail for all compact manifolds and manifolds with a rational curve; in particular, for projective manifolds.
\end{abstract}

\maketitle

\section{Introduction and Results} 
\label{sec:intro}

\noindent
The past 15 years or so have seen the development of a rich theory of approximation and interpolation of holomorphic maps from Stein spaces to complex manifolds that are \lq\lq big\rq\rq\ in the sense that the complex plane is big and the disc is small.  The prototypical examples of \lq\lq big\rq\rq\ complex manifolds are complex Lie groups.  By the 1960s, good approximation and interpolation theorems for them and their homogeneous spaces had been proved by Grauert, Cartan, and others.  In a seminal paper of 1989 \cite{Gromov1989}, Gromov showed how to extend such theorems to the larger class of elliptic manifolds, using his linearisation method of dominating sprays.  Since 2000, the theory has grown into a subfield of holomorphic geometry in its own right, the foremost contributor being Forstneri\v c.  Oka manifolds have emerged as the natural targets of maps from Stein spaces.  They are defined by close to 20 nontrivially equivalent properties involving approximation or interpolation or both.  (The monograph \cite{Forstneric2017} is a comprehensive reference on Oka theory; see also the survey \cite{FL2011}.)

Among the properties of a complex manifold $Y$ investigated in Oka theory are the following, where $X$ denotes an arbitrary reduced Stein space.
\begin{itemize}
\item  \textit{The approximation property} (AP):  Every continuous map $X\to Y$ that is holomorphic on a holomorphically convex compact subset $K$ of $X$ can be uniformly approximated on $K$ by holomorphic maps $X\to Y$.
\item  \textit{The interpolation property} (IP):  A holomorphic map from a subvariety of $X$ to $Y$ has a holomorphic extension $X\to Y$ if it has a continuous extension.
\item  \textit{The basic Oka property} (BOP):  Every continuous map $X\to Y$ is homotopic to a holomorphic map.
\item  \textit{The homotopy Runge property} (HRP):  For every holomorphic map $f_0:X\to Y$, a holomorphically convex compact subset $K$ of $X$ with a neighbourhood $U$, and a homotopy of holomorphic maps $f_t:U\to Y$, $t\in [0,1]$, there is a holomorphic map $F:X\times\C\to Y$ with $F(\cdot,0)=f_0$ and $F(\cdot, t)$ as close to $f_t$ as desired, uniformly on $K$.
\item  \textit{Subellipticity} (SEll):  $Y$ admits a finite dominating family of holomorphic sprays (for more details, see \cite[Definition 5.6.13]{Forstneric2017}).
\item  $\textrm{Ell}_1$:\footnote{This is Gromov's term.  The property could also be called \textit{relative ellipticity}.}  For every holomorphic map $f:X\to Y$, there is a holomorphic map $F:X\times\C^m\to Y$ for some $m\geq 1$, such that $F(\cdot,0)=f$ and $F(x,\cdot):\C^m\to Y$ is a submersion at $0$ for every $x\in X$.
\end{itemize}

Deep theorems provide the following implications.

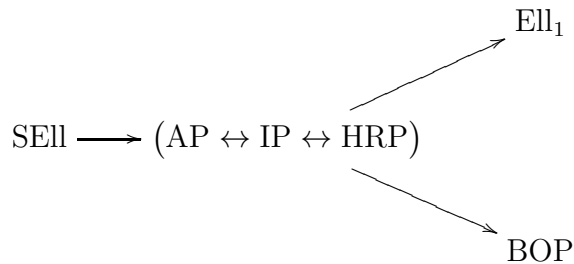
\[ \xymatrix{
  & & \textrm{Ell}_1 \\
 \textrm{SEll} \ar[r] & \big(\textrm{AP}\leftrightarrow\textrm{IP}\leftrightarrow\textrm{HRP}\big) \ar[ur] \ar[dr] & \\
  &  & \textrm{BOP}
} \]
AP, IP, and HRP are among the many equivalent formulations of the Oka property.  Subellipticity is a useful geometric sufficient condition for the Oka property to hold.  $\textrm{Ell}_1$ is primarily of interest as a stepping stone on the way to transversality theorems for holomorphic maps into Oka manifolds (see \cite[Section 8.8]{Forstneric2017}).  The converses of the implications $\rightarrow$ and $\nearrow$ are true when $Y$ is Stein, but are open in general.  The converse of $\searrow$ obviously fails when $Y$ is the disc, but no noncontractible counterexamples are known.  

The goal of this paper is to investigate the analogues of these properties and their relationships in the algebraic category.  Each of the six properties has an algebraic version for an algebraic manifold\footnote{An algebraic manifold is a connected smooth algebraic variety over $\C$, by definition quasi-compact in the Zariski topology.  We take a subvariety to be closed and not necessarily irreducible.} $Y$, with $X$ replaced by an arbitrary affine variety, holomorphic maps from $X$, $X\times\C$, $X\times\C^m$, or subvarieties of $X$ by regular maps (that is, morphisms), and holomorphic sprays on $Y$ by algebraic sprays.  The algebraic analogue of a holomorphic property P will be called aP.

A very different picture emerges in the algebraic case.

\begin{theorem}   \label{t:first}
For algebraic manifolds, aSEll, $\textrm{aEll}_1$, and aHRP are equivalent.
\end{theorem}

\begin{theorem}   \label{t:second}
Let $Y$ be an algebraic manifold.\footnote{We take $Y$ to have positive dimension, that is, we exclude the point.}

{\rm (a)}  If $Y$ contains a rational curve, that is, there is a nonconstant regular map from $\P_1$ to $Y$, then $Y$ does not satisfy aAP, aIP, or aBOP.

{\rm (b)}  If $Y$ is compact, then $Y$ does not satisfy aAP, aIP, or aBOP.
\end{theorem}

Theorem \ref{t:first} suggests that \lq\lq algebraic Oka theory\rq\rq\ should focus on the equivalent properties aSEll, $\textrm{aEll}_1$, and aHRP.  It is tempting to introduce the term \textit{the algebraic Oka property} (aOka) for them.

On the other hand, the properties aAP, aIP, and aBOP are of no interest for compact manifolds and manifolds with a rational curve; in particular for projective manifolds.  For affine manifolds $Y$, the authors' understanding of these properties is limited.  The affine spaces $\C^n$, $n\geq 1$, satisfy the six algebraic properties.  More generally, when $Y$ is contractible, aBOP is obviously true, and aAP holds if, and aIP holds if and only if, $Y$ is a regular retract of some affine space.  We would not be surprised if the three properties turned out to fail for all noncontractible affine manifolds.  We provide some examples below.

The properties aAP, aIP, and aBOP make sense, as defined, for singular varieties.  The proof of Theorem \ref{t:second} is easily generalised to a possibly singular algebraic variety $Y$ that embeds as a subvariety in a smooth variety.  It is well known that not all singular varieties do.  In particular, Theorem \ref{t:second} holds for a projective variety $Y$.  Little is known about the Oka theory of singular targets.  The first paper on this topic is \cite{LL2016}.  The results there show that analytic Oka theory changes quite dramatically when we move from smooth targets to singular targets.

Our third theorem provides a new class of examples of aOka manifolds.

\begin{theorem}   \label{t:toric}
Every smooth nondegenerate toric variety is locally flexible and hence algebraically Oka.
\end{theorem}

It is known that every smooth toric variety is Oka (\cite{Larusson2011}, \cite[Theorem 2.17]{Forstneric2013}).

\begin{remark}   \label{r:forstneric}
We rely on Forstneri\v c's work in \cite{Forstneric2006} (see also \cite[Sections 6.15 and 8.8]{Forstneric2017}).  He proved that aSEll implies both aHRP \cite[Theorem 3.1]{Forstneric2006} and $\textrm{aEll}_1$ \cite[Proposition 4.6]{Forstneric2006}.

Forstneri\v c showed that if a holomorphic map from an affine variety $X$ to an aOka manifold is homotopic to a regular map, then it is approximable by regular maps.  The converse is easily proved, because two continuous maps from $X$ that are sufficiently close on a sufficiently large compact subset of $X$ are homotopic.  (Here we need to know that there is a compact subset of $X$ that is a strong deformation retract of $X$ \cite[Theorem 1.1]{HM1997}.)  For the same reason (noting also that aAP trivially implies AP, which, as already mentioned, nontrivially implies BOP), aAP implies aBOP.

Forstneri\v c gave two aOka counterexamples to aBOP, and therefore to aAP, in \cite[Examples 6.15.7 and 6.15.8]{Forstneric2017}.  One is the complex projective space $\P_n$ for $n\geq 3$.  The other is the complex $n$-sphere $\Sigma^n=\{(z_0,\ldots,z_n)\in\C^{n+1}:z_0^2+\cdots+z_n^2=1\}$ for even $n\geq 2$.  It is a homogeneous space of the connected linear algebraic group $\textrm{SO}(n+1, \C)$, which has no nontrivial characters, and therefore flexible \cite[Proposition 5.4]{AFKKZ2013a}.  

An affine manifold is flexible if its tangent bundle is generated by complete regular vector fields with regular flows.  Equivalently, the subgroup of the algebraic automorphism group generated by subgroups isomorphic to $(\C, +)$ acts infinitely transitively (\cite{AFKKZ2013a}, \cite{AKZ2012}).  The notion of flexibility has been extended to quasi-affine manifolds \cite{FKZ2016}.  An algebraic manifold is locally flexible if it is covered by quasi-affine Zariski-open subsets that are flexible; it is then aOka \cite[Corollary 3.2]{KKT} and has \lq\lq many\rq\rq\ birational automorphisms.
\end{remark}

\begin{remark}
(a)  If an algebraic manifold is aOka, then it is Oka as a complex manifold (because algebraic subellipticity obviously implies subellipticity, which in turn implies the Oka property).

(b)  It is easily seen that:
\begin{itemize}
\item  the product of two aOka manifolds is aOka,
\item  a regular retract of an aOka manifold is aOka,
\item  a finite unbranched covering space of an aOka manifold is aOka.
\end{itemize}
By Gromov's localisation principle for algebraic subellipticity (\cite[\S 3.5.B]{Gromov1989}; see also \cite[Proposition 6.4.2]{Forstneric2017}), the algebraic Oka property is Zariski-local.  (Since the properties aAP, aIP, and aBOP fail for $\P_1$ but hold for $\C$, they are not Zariski-local.)

(c)  A smooth compact algebraic surface $Y$ is aOka if and only if it is rational.  Indeed, if $Y$ is aOka, then $Y$ is unirational and hence rational.  Conversely, if $Y$ is rational, then $Y$ is covered by Zariski-open subsets isomorphic to $\C^2$, so $Y$ is aOka.  (By \cite[Example 5.3]{AFKKZ2013a}, a flexible affine manifold need not be rational or even stably rational.)

(d)  Forstneri\v c showed that every compact aOka manifold of dimension $n$ is the image of a regular map from $\C^n$ (his result \cite[Theorem 1.6]{Forstneric2016} in fact says more).  His argument can be easily extended to a proof that if $Y$ is an $n$-dimensional aOka manifold and $K$ is a compact subset of $Y$, then there is a regular map from $\C^n$ to $Y$ whose image contains $K$.  It follows that every finite subset of $Y$ lies in a regular image of $\C$.  Equivalently, the values of regular maps from affine varieties to $Y$ can be prescribed at finitely many points.

Suppose that $Y$ is an aOka manifold of dimension $n$.  By $\textrm{aEll}_1$, $Y$ is strongly algebraically dominable, meaning that for every $y\in Y$, there is a regular map $g:\C^n \to Y$ with $0\mapsto y$ that is nondegenerate at $0$.  The image of $g$ is constructible and has nonempty interior in the Hausdorff topology, so it contains a nonempty Zariski-open set.  Being quasi-compact in the Zariski topology, $Y$ is covered by the images of finitely many such maps.  If $Y$ is noncompact, we do not know whether $Y$ is the image of a single regular map from $\C^n$.  In particular, we do not know whether $\C^2\setminus\{0\}$ is a regular image of $\C^2$.

(e)  As far as the authors are aware, no aOka manifold is known not to be locally flexible.  Quite a few different kinds of examples of locally flexible manifolds may be found in the literature.  Some are flexible (for a list of examples, see \cite[Section 3]{AFKKZ2013b}) and some are Zariski-locally isomorphic to affine space (also said to be of Class $\mathscr A_0$ or A-covered; for a list of examples, see \cite[Section 4]{APS2014}).  Our Theorem \ref{t:toric} provides a new class of examples of aOka manifolds: smooth nondegenerate toric varieties.

It does not seem reasonable to conjecture that all aOka manifolds are locally flexible.  In the light of our current knowledge, the two classes appear rather different.  The aOka property is preserved by regular retracts; flexibility is not (see Remark \ref{r:affine-counterexamples}(b) below), but for local flexibility it is an open question.  Arbitrary blow-ups of, say, $\C^3$ are aOka \cite[Theorem 1]{LT2017}, but are not known to be locally flexible.  Local flexibility is preserved by removing subvarieties of codimension at least $2$ \cite[Theorem 1.1]{FKZ2016} and yields many birational automorphisms; both are unknown for the aOka property (although aOka manifolds do have many dominant rational self-maps).

(f)  We do not know whether the algebraic Oka property is a birational invariant, but we may be close: the blow-up of a locally flexible (or merely locally stably flexible) algebraic manifold along any algebraic submanifold (not necessarily connected or of pure dimension) is aOka \cite[Theorem 0.3]{KKT}.  Also, strong algebraic dominability, a property of algebraic manifolds that is implied by the algebraic Oka property, is preserved by blowing up along any algebraic submanifold \cite[Theorem 9]{LT2017}.  (Although $\C^2$ satisfies aAP, aIP, and aBOP, $\C^2$ blown up at a point satisfies none of them by Theorem \ref{t:second}(a).)

(g)  If a projective manifold is aOka, then it is unirational.  By Ishkovskikh and Manin's solution of the L\"uroth problem \cite{IM1971}, every smooth quartic in $\P_4$ has a finite group of birational automorphisms, so it cannot be rational or locally flexible, whereas Segre showed that some quartics are unirational.  It follows that among projective manifolds (in fact among smooth quartics in $\P_4$), either there are aOka manifolds that are not locally flexible, or there are unirational manifolds that are not aOka (or both).
\end{remark}

\begin{remark}  \label{r:affine-counterexamples}
(a)  Let $Y$ be a projective variety.  By the Jouanolou trick (first used in Oka theory in \cite{Larusson2005}), $Y$ carries an affine bundle whose total space $A$ is affine.  For $Y=\P_m$, we take $A$ to be the complement $A_m$ in $\P_m\times\P_m$ of the hypersurface defined by the equation $z_0w_0+\cdots+z_mw_m=0$, with the projection $A_m\to\P_m$ onto the first component.  In general, we embed $Y$ into $\P_m$ for some $m$ and let $A$ be the pullback of $A_m$ by the inclusion $Y\hookrightarrow\P_m$.  We sometimes call $A$ an affine model or a Stein model for $Y$.  Every holomorphic map from a reduced Stein space to $Y$ factors holomorphically (not necessarily uniquely) through $A$.  We claim that $A$ fails to satisfy aIP.

By the proof of Theorem \ref{t:second}, the failure of $Y$ to satisfy aIP is demonstrated by the sources $\{0,1\}\hookrightarrow \C$ or by the sources $S\hookrightarrow \C^2$, where $S$ is a smooth irrational curve.  In the former case, we easily deduce that $A$ fails to satisfy aIP.  In the latter case, there is a nullhomotopic regular map $f:S\to Y$ that does not factor regularly through $\C^2$ although it does continuously.  We claim that $f$ has a regular lifting to $A$.  The lifting is also nullhomotopic and does not factor regularly through $\C^2$ either, so $A$ does not satisfy aIP.  

To prove that $f$ has a regular lifting to $A$, it suffices to take $Y=\P_m$ and $A=A_m$.  We need to show that if $f:S\to\P_m$ is a regular map from a smooth affine curve $S$, then there is a regular map $g:S\to\P_m$ such that $(f,g)$ avoids the hypersurface in $\P_m\times\P_m$ defined by the equation $z_0w_0+\cdots+z_mw_m=0$.  Write $f=[f_0,\ldots,f_m]$, where $f_0,\ldots,f_m$ are regular functions on $S$, possibly with a common zero set $Z$ that cannot be eliminated.  Let the divisor $D$ on $S$ be the minimum of the divisors of $f_0,\ldots,f_m$ and consider the short exact sequence
\[ 0 \to \operatorname{Ker}\beta \to \mathscr O_D^{m+1} \overset\beta\to \mathscr O \to 0, \]
on $S$, where $\beta(g_0,\ldots,g_m)=f_0g_0+\cdots+f_mg_m$.  Since $H^1(S,\operatorname{Ker}\beta)=0$, there are $g_0,\ldots,g_m\in\mathscr O_D(S)$ with $f_0g_0+\cdots+f_mg_m=1$.  Let $z$ be a coordinate centred at a point of $Z$ where $D=k\geq 1$.  Near the point,
\[ (z^{-k}f_0)(z^k g_0)+\cdots+(z^{-k}f_m)(z^k g_m) = 1. \]
We conclude that the regular map $g=[g_0,\ldots,g_m]:S\to\P_m$ is as desired.

(b)  We have already noted that a flexible affine manifold need not satisfy aBOP, let alone aAP.  Now we present a flexible affine counterexample to aIP.

By a simple change of coordinates, the affine model $A_1$ of $\P_1$ can be realised as the affine surface
\[ A_1 = \{(x,y,z)\in\C^3:xy=z(1-z) \}. \]
Then the projection $A_1\to\P_1$ takes $(x,y,z)$ to $[x,z]=[1-z,y]$.  A Danielewski surface is a smooth affine surface of the form
\[ D_n = \{(x,y,z)\in\C^3:x^n y=p(z) \}, \]
where $n\geq 1$ and $p$ is a polynomial of degree at least $2$ all of whose zeros are simple.  So $A_1$ is of the form $D_1$.  It is known that Danielewski surfaces with $n=1$ are flexible, whereas for $n\geq 2$ they are locally flexible but not flexible.  Moreover, for the same $p$, the surfaces $D_n\times \C$ are mutually isomorphic for all $n\geq 1$.  Thus $A_1$ is flexible.  It also follows that for affine manifolds, flexibility is not preserved by regular retracts, and local flexibility does not imply flexibility.

By another simple change of coordinates, we can realise $A_1$ as the complex $2$-sphere $\Sigma^2$, showing again that $A_1$ is flexible (see Remark \ref{r:forstneric}).
\end{remark}

\section{Proofs of the theorems}
\label{sec:proofs}

\begin{proof}[Proof of Theorem \ref{t:first}]
As mentioned above, aSEll implies both aHRP \cite[Theorem 3.1]{Forstneric2006} and $\textrm{aEll}_1$ \cite[Proposition 4.6]{Forstneric2006}.  We merely observe that aHRP and $\textrm{aEll}_1$ both easily imply the following weaker property of an algebraic manifold $Y$, which in turn clearly implies aSEll by the powerful localisation principle.

\noindent
\textit{Weak formulation of aOka}:  For every $a\in Y$, the tangent space $T_a Y$ can be spanned by vectors $v$, such that there is a Zariski-open neighbourhood $U$ of $a$ in $Y$ (which might as well be taken to be affine) and a regular map $f:U\times\C\to Y$ with $f(y,0)=y$ for all $y\in U$ and $D_0f(a,\cdot)\dfrac d{dz}=v$.
\end{proof}

Next we turn to Theorem \ref{t:toric}.  Our proof relies on two theorems.
\begin{itemize}
\item  The smooth locus of a nondegenerate affine toric variety is flexible \cite[Theorem 0.2]{AKZ2012}.
\item  The complement of a subvariety of codimension at least $2$ in a flexible quasi-affine manifold is flexible \cite[Theorem 1.1]{FKZ2016}.
\end{itemize}

\begin{proof}[Proof of Theorem \ref{t:toric}]
Let $Y$ be a smooth nondegenerate toric variety.  It is defined by a fan $F$ of cones in a vector space $N_\mathbb R=N\otimes_{\mathbb Z}\mathbb R$, where $N$ is a lattice.  The fan is smooth, meaning that the minimal generators of each cone in $F$ form part of a $\mathbb Z$-basis for $N$.  Nondegeneracy of $Y$ means that the minimal generators of all the cones together span~$N_\mathbb R$.

Each maximal cone in $F$ defines an affine toric Zariski-open subset of $Y$, and these subsets cover $Y$.  If the dimension of the cone is $n = \dim Y$, then the corresponding subset is $\C^n$ (because the cone is smooth), which is flexible.  If the dimension of the cone $C$ is $k<n$, then the corresponding subset $U$ is $\C^k \times (\C^*)^{n-k}$.  We will extend $U$ to a flexible quasi-affine toric Zariski-open subset $V$ of $Y$.  This will complete the proof.

We define $V$ by a subfan $F'$ of $F$.  The subfan $F'$ contains $C$ along with the $1$-dimensional cones spanned by some of the minimal generators of the other cones in $F$, so that these generators, together with the minimal generators of $C$, form a basis for~$N_\mathbb R$.

Let $C'$ be the cone spanned by the cones in $F'$, that is, spanned by the minimal generators of $C$ and the additional minimal generators used to define $F'$.  The cone $C'$ is pointed (that is, strictly convex), $n$-dimensional, and defines a possibly singular nondegenerate affine toric variety $Z$.  All the edges of $C'$ are contained in $F'$, so by the orbit-cone correspondence, $V$ is realised as the complement in $Z$ of a toric subvariety of codimension at least $2$.  By the two theorems, $V$ is flexible.
\end{proof}

Without the two theorems, our proof shows that $Y$ has a Zariski-open cover by two kinds of sets.  The first kind is just $\C^n$, coming from a maximal cone $C$ of full dimension.  A set $V$ of the second kind is the complement of a subvariety of codimension at least $2$ in the possibly singular nondegenerate affine toric variety $Z$.  By \cite[Exercise 1.2.10 and Example 1.3.20]{CLS2011}, $Z$ is simplicial, so it is of the form $\C^n/G$, where $G$ is a finite abelian group.  Therefore $V$ is a finite unbranched Galois quotient of the complement of a subvariety of codimension at least $2$ in $\C^n$.  If we could prove directly that such a set $V$ was flexible or just aOka, then we would not have to invoke the two theorems.

Now we turn to the proof of Theorem \ref{t:second}.  We first consider the case of projective varieties.  The proof relies on the failure of the GAGA principle for line bundles on affine manifolds.  We show that a projective variety fails to satisfy aAP, aIP, and aBOP for some very particular sources.

\begin{proof}[Proof of Theorem \ref{t:second} for projective varieties]
Let $Y$ be a projective variety.  First, note that if $Y$ satisfies aIP, then we can use the inclusion $\{0,1\}\hookrightarrow\C$ to obtain a nonconstant regular map $\P_1\to Y$.  Second, suppose that $Y$ satisfies aBOP.  Since $Y$ is not contractible, $\pi_k(Y)\neq 0$ for some $k\geq 1$, so there is a continuous map $\Sigma^k\to Y$ that is not homotopic to a constant map.  (Here, $\Sigma^k$ is the complex $k$-sphere defined in Remark \ref{r:forstneric} above; it contains and retracts onto the real $k$-sphere.)  Then, by aBOP, there is a nonconstant regular map $\Sigma^k\to Y$.  One-parameter subgroups of the linear algebraic group $\textrm{SO}(k+1, \C)$ give many regular maps $\C^*\to\Sigma^k$, so again there is a nonconstant regular map $\P_1\to Y$.

We conclude that if $Y$ does not contain a rational curve, then $Y$ fails to satisfy aIP and fails to satisfy aBOP, and hence fails to satisfy aAP.  We continue the proof assuming that there is a nonconstant regular map $g:\P_1\to Y$.

Let $S$ be a smooth irrational curve in $\C^2$.  It is well known that the algebraic Picard group of $S$ is \lq\lq large\rq\rq, even though its holomorphic Picard group is trivial.  The algebraic Picard group has plenty of nontorsion elements, so there is an algebraic line bundle $L$ on $S$ so that no nonzero tensor power of $L$ is algebraically trivial.\footnote{We have $S=M\setminus F$, where $M$ is a smooth projective curve of genus at least $1$ and $F\subset M$ is finite and nonempty.  Every algebraic line bundle on $S$ extends to $M$, so the algebraic Picard group of $S$ is a quotient of the Picard group of $M$ by the finitely-generated subgroup corresponding to divisors with support in $F$.}  It is generated by two regular sections, so it is the pullback of the universal bundle on $\P_1$ by a regular map $f:S\to\P_1$.  Let $P$ be an ample line bundle on $Y$.  Then $g^*P$ is ample on $\P_1$, so $f^*g^*P$ is a nonzero tensor power of $L$ and hence algebraically nontrivial.  Therefore, by the Quillen-Suslin theorem, $g\circ f:S\to Y$ does not factor regularly through $\C^2$, even though it does continuously because $f$ is nullhomotopic.  This shows that $Y$ does not satisfy aIP.

To show that $Y$ does not satisfy aBOP, we use the fact that $\C^*\times\C^*$ has a \lq\lq large\rq\rq\ holomorphic Picard group, isomorphic to $H^2(\C^*\times\C^*, \mathbb Z)$, even though its algebraic Picard group is trivial.  Let $L$ be a nontrivial holomorphic line bundle on $\C^*\times\C^*$.  Then no nonzero tensor power of $L$ is holomorphically trivial.  As we will explain in a moment, $L$ is generated by two holomorphic sections, so it is the pullback of the universal bundle on $\P_1$ by a holomorphic map $f:\C^*\times\C^*\to\P_1$.  As before, let $P$ be an ample line bundle on $Y$.  Then $g^*P$ is ample on $\P_1$, so $f^*g^*P$ is a nonzero tensor power of $L$ and hence holomorphically nontrivial.  If $g\circ f:\C^*\times\C^*\to Y$ could be deformed to a regular map $h$, then $h^*P$ would be algebraically and hence topologically trivial, so $f^*g^*P$ would be topologically trivial as well, and hence holomorphically trivial by Grauert's Oka principle.  This shows that $Y$ fails to satisfy aBOP and hence aAP.

To generate $L$ by two holomorphic sections on $\C^*\times\C^*$, we first choose a nontrivial section $s$, whose zero locus is a $1$-dimensional subvariety $Z$ of $\C^*\times\C^*$.  Now $Z$ has the homotopy type of a union of bouquets of circles (see \cite{HM1997}), so $L\vert_Z$ is topologically and hence holomorphically trivial.  Take a holomorphic section of $L\vert_Z$ without zeros and extend it to a holomorphic section $t$ of $L$.  Then $s$ and $t$ generate $L$.
\end{proof}

To prove Theorem \ref{t:second} in full generality we need a lemma, probably well known to experts, but for want of a reference we sketch a proof.

\begin{lemma}    \label{l:lemma-1}
Let $Y$ be a compact algebraic manifold and $C$ be an irreducible curve in $Y$.  There is a finite composition $Y'\to Y$ of blow-ups with smooth centres, such that $Y'$ is projective and $C$ is not contained in the image of any of the exceptional divisors.
\end{lemma}

\begin{proof}
Let $U$ be an affine Zariski-open subset of $Y$ with $U\cap C\neq\varnothing$.  By a strong version of Chow's lemma \cite[Proposition 5]{Raynaud1972}, there is a blow-up $p:\tilde Y\to Y$ along an ideal $I$ cosupported on $Y\setminus U$, such that $\tilde Y$ is smooth and projective.  We can then use the arguments in \cite[proof of Proposition D]{Moishezon1967} to finish the proof as follows.  By Hironaka's resolution of singularities, there is an iterated blow-up $\pi:Y'\to Y$ with smooth centres over $Y\setminus U$, such that the ideal $\pi^*I$ is principal.  Then $Y'$ is smooth and compact and by the universal property of blow-ups, there is a birational morphism $g:Y'\to \tilde Y$ such that $\pi=p\circ g$.  Since $p$ is a blow-up and hence a modification, $g$ is also a blow-up by \cite[Lemma 4]{HR1964}.  Since $\tilde Y$ is projective, so is $Y'$.  The exceptional divisors of $\pi$ lie over $Y\setminus U$, so their images do not contain $C$.
\end{proof}

\begin{proof}[Proof of Theorem \ref{t:second}]
(a)  Let $Y$ be an algebraic manifold with a nonconstant regular map $f:\P_1\to Y$.
 By Nagata's compactification theorem and Hironaka's resolution of singularities, there is a smooth compactification $\bar Y$ of $Y$.  By Lemma \ref{l:lemma-1}, there is a finite composition $\pi:Y'\to \bar Y$ of blow-ups with smooth centres, such that $Y'$ is projective and $f(\P_1)$ is not contained in the image of any of the exceptional divisors.  
 
Embed $Y'$ in a projective space $\P_m$.  Let $L$ be the hyperplane bundle on $\P_m$ and take a hyperplane $H$ in $\P_m$ that does not contain $\pi^{-1}(f(\P_1))$.  Then $\deg f^*\pi_*(L\vert_{Y'})$ equals the intersection number $f_*(\P_1)\cdot \pi_*(L\vert_{Y'})$, that is, the intersection number of the effective curve $f_*(\P_1)$ and the effective divisor $\pi_*(H\cap Y')$.  This number is positive since $f(\P_1)$ is irreducible and not contained in $\pi(H\cap Y')$, so $f^*\pi_*(L\vert_{Y'})$ is an ample line bundle on $\P_1$.  Now we can proceed as in the proof for a projective variety $Y$.
 
(b)  If $Y$ is a compact algebraic manifold satifying aAP, aIP, or aBOP, then, as in the proof for projective varieties, we can show that $Y$ has a rational curve.  Then we invoke part (a).
\end{proof}

\end{document}